\documentclass[11pt]{article}
\usepackage[a4paper]{geometry}
\usepackage{amssymb,latexsym,amsmath,amsfonts,amsthm}
\usepackage{graphicx,overpic}
\usepackage{dsfont}
\usepackage{mathrsfs}
\usepackage{enumerate}
\usepackage{epsfig}
\usepackage{epstopdf}
\usepackage{booktabs}
\usepackage{mathtools}
\usepackage{multirow}
\usepackage{pifont}
\usepackage{comment,verbatim}
\usepackage{hyperref}
\usepackage{bbm}
\usepackage[toc,page]{appendix}

\newtheorem{theorem}{Theorem}[section]
\newtheorem{lemma}[theorem]{Lemma}

\newtheorem{corollary}[theorem]{Corollary}

\theoremstyle{definition}

\newtheorem{example}[theorem]{Example}

\theoremstyle{remark}

\newtheorem{remark}[theorem]{Remark}

\numberwithin{equation}{section}

\usepackage{color}

\usepackage[normalem]{ulem}

\begin{document}
\title{On the deterioration of convergence rate of spectral differentiations for functions with singularities}
\author{Haiyong Wang\footnotemark[1]~\footnotemark[2]}
\date{}
\maketitle

\footnotetext[1]{School of Mathematics and Statistics, Huazhong
University of Science and Technology, Wuhan 430074, P. R. China.
E-mail: \texttt{haiyongwang@hust.edu.cn}}

\footnotetext[2]{Hubei Key Laboratory of Engineering Modeling and
Scientific Computing, Huazhong University of Science and Technology,
Wuhan 430074, P. R. China}

\begin{abstract}
Spectral differentiations are basic ingredients of spectral methods. In this work, we analyze the pointwise rate of convergence of spectral differentiations for functions containing singularities and show that the deteriorations of the convergence rate at the endpoints, singularities and other points in the smooth region exhibit different patterns. As the order of differentiation increases by one, we show for functions with an algebraic singularity that the convergence rate of spectral differentiation by Jacobi projection deteriorates two orders at both endpoints and only one order at each point in the smooth region. The situation at the singularity is more complicated and the convergence rate either deteriorates two orders or does not deteriorate, depending on the parity of the order of differentiation, when the singularity locates in the interior of the interval and deteriorates two orders when the singularity locates at the endpoint. Extensions to some related problems, such as the spectral differentiation using Chebyshev interpolation, are also discussed. Our findings justify the error localization property of Jacobi approximation and differentiation and provide some new insight into the convergence behavior of Jacobi spectral methods.
\end{abstract}


\vspace{0.05in}


\section{Introduction}\label{sec:introduction}
Spectral methods play an important role in the simulation of
differential equations arising in mathematics and physics. One of the most attractive
advantages of them is that they have the so-called spectral
accuracy, that is, their convergence rate depends solely on the
regularity of the underlying functions. Let $\Omega:=[a,b]$ and let $\omega(x)\geq0$ be a weight function. We introduce an inner product and the associated norm
\begin{equation}\label{def:InnerProd}
\langle f,g \rangle = \int_{\Omega}f(x)g(x)\omega(x)\mathrm{d}x,
\quad \|f\| = \sqrt{\langle f,f \rangle}.
\end{equation}
Let $\{\phi_k\}_{k=0}^{\infty}$ denote the sequence of orthogonal
polynomials respect to the above inner product, i.e., $\langle
\phi_j,\phi_k \rangle = \gamma_k\delta_{kj}$, where $\gamma_k>0$ and $\delta_{kj}$ is
the Kronnecker delta. Let $L_{\omega}^2(\Omega)$ denote the space of functions that are square integrable with respect to the inner product
\eqref{def:InnerProd} and let $\mathbb{P}_n$ denote the space of
polynomials of degree at most $n$, i.e.,
$\mathbb{P}_n=\mathrm{span}\{x^k\}_{k=0}^{n}$. For any
$f\in{L}_{\omega}^2(\Omega)$, the orthogonal projection from
${L}_{\omega}^2(\Omega)$ upon $\mathbb{P}_n$ can be written as
\begin{align}\label{eq:OrthProj}
(\Pi_n f)(x) = \sum_{k=0}^{n} a_k \phi_k(x), \quad  a_k = \frac{\langle
f,\phi_k \rangle}{\gamma_k}.
\end{align}
In the past decade, weighted and maximum error estimates of classical spectral approximations, including Jacobi, Laguerre and Hermite approximations, have received considerable research interest among spectral method community and their sharp results for analytic and differentiable functions are nowadays well understood.

Recently, pointwise error estimates of classical spectral approximations for functions with singularities have attracted considerable attention (see, e.g., \cite{Babuska2019,Trefethen2011,Wang2021,Wang2023a,Wang2023b,Xiang2023}). One motivation comes from the error behaviors of best approximation and Chebyshev interpolant. Specifically, Trefethen in \cite{Trefethen2011} compared the pointwise errors of best approximation and Chebyshev interpolant for $f(x)=|x-1/4|$ and observed that the accuracy of the former is better than that of the latter only in a small neighborhood of the singularity $x=1/4$. This observation was listed as the third of the seven myths on polynomial interpolation and quadrature and it shows that Chebyshev interpolant has the error localization property, that is, its maximum error is always dominated by the error at the singularity. How to understand this phenomenon? The present author in \cite{Wang2023a} studied this myth and proved that the convergence rate of Chebyshev approximations (including projection and interpolation) at each point away from the singularity is actually one order higher than that of at the singularity. This provides a rigorous justification for Trefethen's observation. Another motivation for the study of pointwise error estimates is to establish sharp error estimates for classical spectral approximations in the maximum norm. Consider the case of Legendre projections, i.e., $\phi_k(x)=P_k(x)$, where $P_k(x)$ is the Legendre polynomial of degree $k$. In view of the inequality $|P_k(x)|\leq1$, it is easily seen that the maximum error of Legendre projection satisfies
\begin{equation}
\|f - \Pi_nf\|_{\infty} \leq \sum_{k=n+1}^{\infty} |a_k|.
\end{equation}
Therefore, once the estimate of the Legendre coefficients $\{a_k\}$ was available, the error estimate of Legendre projection in the maximum norm can be established immediately. However, it was observed in \cite{Wang2018} that there is a half order loss in the maximum error estimate of Legendre projection for the example $f(x)=|x|$ even though the estimate of the Legendre coefficients is sharp. Why we lose half order? The key reason is that the maximum error of Legendre projection is attained at the singularity $x=0$ (see, e.g., \cite{Liu2021,Wang2021,Wang2023a}). However, the inequality $|P_k(x)|\leq1$ is sharp only when $x=\pm1$, i.e., $|P_k(\pm1)|=1$, but becomes pessimistic when $x\in(-1,1)$ since $|P_k(x)|=O(k^{-1/2})$ as $k\rightarrow\infty$. Clearly, this issue also highlights the necessity of studying pointwise error estimates of classical spectral approximations. More recently, Babu\v{s}ka and Hakula in \cite{Babuska2019} studied the pointwise error estimates of Legendre projection for the truncated power function $f(x)=(x-\xi)_{+}^{\sigma}$ and they investigated the preasymptotic and asymptotic convergence rates of Legendre projection at each point $x\in[-1,1]$. Xiang etc in \cite{Xiang2023} further established the pointwsie rates of convergence of Jacobi projection for the truncated power functions and functions with an algebraic singularity.

In this work, we continue the research initiated in \cite{Wang2023a} and systematically analyze the pointwise rate of convergence of spectral differentiation for functions containing singularities. An interesting finding in \cite[Theorem 3.3]{Wang2023a} is that the convergence rate of the first-order spectral differentiation by Chebyshev projection at the interior algebraic singularity does not deteriorate. It is curious to ask how does spectral differentiation deteriorate its convergence rate at each point of the approximation interval? Here we will explore this issue thoroughly and show that the deteriorations of convergence rate of spectral differentiation at the endpoints, singularities and other points in the smooth region exhibit quite different patterns. Specifically, for functions with an algebraic singularity, the convergence rate of spectral differentiation by Jacobi projection deteriorates two orders at both endpoints and only one order at each point in the smooth region as the order of differentiation increases by one. The situation at the singularity is more complicated and the convergence rate deteriorates two orders if the singularity is located at the endpoint and does not deteriorate when $m$ is even and deteriorates two orders when $m$ is odd as the order of differentiation increases from $m$ to $m+1$ if the singularity is located at the interior of the interval. Moreover, we also show that the maximum error of the first- and higher-order spectral differentiations by Jacobi projection is always dominated by the errors at the endpoints, a phenomenon that may be familiar in spectral method community. Finally, we extend our discussion to spectral differentiation by Chebyshev interpolation and Jacobi spectral differentiation for some other singular functions and similar deterioration results are proved except possibly at the interior singularity.


The paper is organized as follows. In the next section, we review some preliminaries that are useful for the subsequent analysis. In Section \ref{sec:Pointwise} we present a thorough analysis of the pointwise error estimates of Jacobi spectral differentiations for functions with an algebraic singularity. In Section \ref{sec:Extension}, we extend our analysis to several related problems. 
In Section \ref{sec:conclusion}, we finish this paper with some concluding remarks.

\section{Pointwise error estimates of Jacobi spectral differentiations for functions with singularities}\label{sec:Pointwise}
Let $P_n^{(\alpha,\beta)}(x)$ be the Jacobi polynomial of degree $n$ defined by
\begin{equation}\label{def:JacPoly}
P_n^{(\alpha,\beta)}(x) = \frac{(\alpha+1)_n}{n!} {}_2\mathrm{
F}_1\left[\begin{matrix} -n, & n+\alpha+\beta+1
\\   \alpha+1  \hspace{-1cm} &\end{matrix} ; ~ \frac{1-x}{2}
\right],
\end{equation}
where ${}_2 \mathrm{F}_1(\cdot)$ is the Gauss hypergeometric
function and $(z)_k$ denotes the Pochhammer symbol defined by
$(z)_{k} = (z)_{k-1}(z+k-1)$ for $k\geq1$ and $(z)_0=1$. Let $\alpha,\beta>-1$, $\Omega:=[-1,1]$ and let $\omega_{\alpha,\beta}(x)=(1-x)^{\alpha}(1+x)^{\beta}$ be the Jacobi weight function. We redefine the inner product \eqref{def:InnerProd} and the associated norm with respect to $\omega_{\alpha,\beta}(x)$ as
\begin{equation}\label{def:JacIP}
\langle f,g \rangle_{\alpha,\beta} = \int_{\Omega} f(x) g(x) \omega_{\alpha,\beta}(x) \mathrm{d}x, \quad
\|f \|_{\alpha,\beta} = \langle f,g \rangle_{\alpha,\beta}^{1/2}.
\end{equation}
It is well known that Jacobi polynomials satisfy the orthogonal relation with respect to the above inner product. More specifically, we have
\begin{equation}\label{eq:JacOrth}
\langle P_m^{(\alpha,\beta)}, P_n^{(\alpha,\beta)} \rangle_{\alpha,\beta} = h_n^{(\alpha,\beta)} \delta_{mn},
\end{equation}
where $\delta_{mn}$ is the Kronecker delta and
\begin{equation}
h_n^{(\alpha,\beta)} = \frac{2^{\alpha+\beta+1} \Gamma(n+\alpha+1) \Gamma(n+\beta+1)}{(2n+\alpha+\beta+1)\Gamma(n+1) \Gamma(n+\alpha+\beta+1)}. \nonumber
\end{equation}
Let $L_{\omega_{\alpha,\beta}}^2(\Omega)$ denote the space of functions that are square integrable with respect to the inner product \eqref{def:JacIP}, i.e., $L_{\omega_{\alpha,\beta}}^2(\Omega) = \left\{f: \|f \|_{\alpha,\beta} < \infty  \right\}$. For any $f\in L_{\omega_{\alpha,\beta}}^2(\Omega)$, the orthogonal projection onto the space $\mathbb{P}_n$ is defined by
\begin{equation}\label{eq:JacAppr}
S_n^{(\alpha,\beta)}(x) = \sum_{k=0}^{n} a_k^{(\alpha,\beta)} P_k^{(\alpha,\beta)}(x), \quad
a_k^{(\alpha,\beta)} = \frac{\langle f,P_k^{(\alpha,\beta)} \rangle_{\alpha,\beta}}{h_k^{(\alpha,\beta)}}.
\end{equation}
where $\{a_k^{(\alpha,\beta)}\}$ are the Jacobi coefficients of $f(x)$. In this section we shall consider the pointwise error estimate of Jacobi approximation and differentiation for functions with singularities. For simplicity of exposition we restrict our attention to the model function
\begin{equation}\label{def:ModelFun}
f(x)=|x-\xi|^{\sigma} g(x),
\end{equation}
where $\xi\in\Omega$ and $g(x)$ is analytic in a neighborhood of $\Omega$ and the exponent $\sigma$ is not an even integer when $\xi\in(-1,1)$ and is not an integer when $\xi=\pm1$. Extension to some other singular functions, such as truncated power functions, will be discussed in the next section.
\begin{remark}
For simplicity of exposition, we assumed that $g(x)$ is analytic in a neighborhood of $\Omega$. This assumption, however, can be relaxed to the case where $g\in{C}^{\nu}(\Omega)$ for some sufficiently large $\nu\in\mathbb{N}$ at the cost of more lengthy and cumbersome proofs.
\end{remark}


\subsection{Asymptotics of the Jacobi coefficients}
Our first result is stated in the following.
\begin{theorem}\label{thm:AsyJac}
Let $f$ be the function defined in \eqref{def:ModelFun} with the exponent $\sigma$ satisfying $\sigma>-1$ when $\xi\in(-1,1)$ and $\sigma>-\beta-1$ when $\xi=-1$ and $\sigma>-\alpha-1$ when $\xi=1$. Let $\{a_k^{(\alpha,\beta)}\}$ denote the Jacobi coefficients of $f$. As $k\rightarrow\infty$, the following results hold.
\begin{itemize}
\item[\rm(i)] If $\xi\in(-1,1)$, then
\begin{align}\label{eq:JacIntAsy}
a_k^{(\alpha,\beta)} &= \mathcal{A}_{\sigma,\xi}^{\alpha,\beta} \frac{\cos(k\arccos(\xi)-\psi_{\alpha,\beta}(\xi))}{k^{\sigma+1/2}} + O\left(\frac{1}{k^{\sigma+3/2}}\right),
\end{align}
where $\psi_{\alpha,\beta}(x)=(2\alpha+1)\pi/4 - (\alpha+\beta+1)\arccos(x)/2$ and
\begin{align}
\mathcal{A}_{\sigma,\xi}^{\alpha,\beta} &= \frac{\Gamma(\sigma+1) (1-\xi)^{\frac{\sigma+\alpha}{2}+\frac{1}{4}} (1+\xi)^{\frac{\sigma+\beta}{2}+\frac{1}{4}} }{ 2^{\frac{\alpha+\beta-3}{2}} \sqrt{\pi} } \cos\left(\frac{\sigma+1}{2} \pi \right) g(\xi). \nonumber
\end{align}

\item[\rm(ii)] If $\xi=-1$, then 
\begin{align}\label{eq:JacEndAsyL}
a_k^{(\alpha,\beta)} &= 
\frac{(-1)^k \mathcal{B}_{\sigma}^{L} }{k^{2\sigma+\beta+1}} + O\left(\frac{1}{k^{2\sigma+\beta+2}}\right), \quad
\mathcal{B}_{\sigma}^{L} = \frac{2^{\sigma+1}\Gamma(\sigma+\beta+1)}{\Gamma(-\sigma)} g(-1).
\end{align}
If $\xi=1$, then 
\begin{align}\label{eq:JacEndAsyR}
a_k^{(\alpha,\beta)} &= 
\frac{\mathcal{B}_{\sigma}^{R}}{k^{2\sigma+\alpha+1}} + O\left(\frac{1}{k^{2\sigma+\alpha+2}}\right), \quad
\mathcal{B}_{\sigma}^{R} = \frac{2^{\sigma+1}\Gamma(\sigma+\alpha+1)}{\Gamma(-\sigma)} g(1).
\end{align}
\end{itemize}
\end{theorem}
\begin{proof}
Taking the Taylor expansion of $g(x)$ at $x=\xi$, we have
\begin{align}
a_k^{(\alpha,\beta)} &= \frac{1}{h_k^{(\alpha,\beta)}} \sum_{\ell=0}^{\infty} \frac{g^{(\ell)}(\xi)}{\ell!} \int_{\Omega} \omega_{\alpha,\beta}(x) P_k^{(\alpha,\beta)}(x) |x-\xi|^{\sigma} (x-\xi)^{\ell} \mathrm{d}x. \nonumber
\end{align}
For the integrals on the right-hand side, we have
\begin{align}
\int_{\Omega} \omega_{\alpha,\beta}(x) P_k^{(\alpha,\beta)}(x) |x-\xi|^{\sigma} (x-\xi)^{\ell} \mathrm{d}x &= (-1)^{\ell} \int_{-1}^{\xi} \omega_{\alpha,\beta}(x) P_k^{(\alpha,\beta)}(x) (\xi-x)^{\sigma+\ell} \mathrm{d}x \nonumber \\
& + \int_{\xi}^{1} \omega_{\alpha,\beta}(x) P_k^{(\alpha,\beta)}(x) (x-\xi)^{\sigma+\ell} \mathrm{d}x. \nonumber
\end{align}
Let $\Upsilon_{\ell}^{(1)}(k)$ and $\Upsilon_{\ell}^{(2)}(k)$ denote respectively the first and second integrals on the right-hand side of the last equation. By virtue of \cite[Equation~(2.22.4.1)]{Prud1986} we obtain after some simplification that 
\begin{align}
\Upsilon_{\ell}^{(1)}(k) &= \mathrm{B}(\beta+1,\sigma+\ell+1) 2^{\alpha} (1+\xi)^{\sigma+\ell+\beta+1} P_k^{(\alpha,\beta)}(-1) \nonumber  \\
&\times {}_2\mathrm{F}_1\left[\begin{matrix} -k-\alpha, & k+\beta+1
\\  \sigma+\ell+\beta+2  \hspace{-1cm} &\end{matrix} ; ~ \frac{1+\xi}{2}
\right], \nonumber
\end{align}
where $\mathrm{B}(\cdot,\cdot)$ is the Beta function, and by the symmetry property of Jacobi polynomials, i.e., $P_n^{(\alpha,\beta)}(x)=(-1)^n P_n^{(\beta,\alpha)}(-x)$, we obtain that
\begin{align}
\Upsilon_{\ell}^{(2)}(k) &= \mathrm{B}(\alpha+1,\sigma+\ell+1) 2^{\beta} (1-\xi)^{\sigma+\ell+\alpha+1} P_k^{(\alpha,\beta)}(1) \nonumber \\
&\times {}_2\mathrm{F}_1\left[\begin{matrix} -k-\beta, & k+\alpha+1
\\  \sigma+\ell+\alpha+2  \hspace{-1cm} &\end{matrix} ; ~ \frac{1-\xi}{2}
\right]. \nonumber
\end{align}
When $\xi\in(-1,1)$, by the asymptotic expansion of Gauss hypergeometric function in \cite{Jones2001} and the asymptotic expansion of the ratio of gamma functions in \cite[Equation~(5.11.13)]{Olver2010}, we derive the desired result \eqref{eq:JacIntAsy}. When $\xi=-1$, we easily seen that $\Upsilon_{\ell}^{(1)}(k)=0$ and
\begin{align}
\Upsilon_{\ell}^{(2)}(k) &= \mathrm{B}(\alpha+1,\sigma+\ell+1) 2^{\beta} (1-\xi)^{\sigma+\ell+\alpha+1} P_k^{(\alpha,\beta)}(1) {}_2\mathrm{F}_1\left[\begin{matrix} -k-\beta, & k+\alpha+1
\\  \sigma+\ell+\alpha+2  \hspace{-1cm} &\end{matrix} ; ~ 1
\right] \nonumber \\
&= \frac{2^{\alpha+\beta+\sigma+\ell+1} \Gamma(\sigma+\ell+\beta+1) \Gamma(\sigma+\ell+1) \Gamma(k+\alpha+1) }{\Gamma(k+1) \Gamma(\sigma+\ell+1-k) \Gamma(k+\sigma+\ell+\alpha+\beta+2)}, \nonumber
\end{align}
where we have used \cite[Equation~(15.4.20)]{Olver2010} in the second step. The result \eqref{eq:JacEndAsyL} then follows by using \cite[Equation~(5.11.13)]{Olver2010}. The proof of the case $\xi=1$ is similar and we omit the details. This ends the proof.
\end{proof}


\begin{remark}
From Theorem \ref{thm:AsyJac} we see that
\[
a_k^{(\alpha,\beta)} = \left\{
            \begin{array}{ll}
{\displaystyle O(k^{-2\sigma-\beta-1})},   & \hbox{$\xi=-1$,}   \\[6pt]
{\displaystyle O(k^{-\sigma-1/2})},   & \hbox{$\xi\in(-1,1)$,}         \\[6pt]
{\displaystyle O(k^{-2\sigma-\alpha-1})},  & \hbox{$\xi=1$,}
            \end{array}
            \right.
\]
and these decay rates have been derived in \cite{Xiang2020} by using the Hilb-type formula of Jacobi polynomials. Note that we only derived the leading terms of the asymptotic expansion of the Jacobi coefficients here, but higher order asymptotic expansions of the Jacobi coefficient can be derived with more calculations. As will be shown in the next subsection, these leading terms in Theorem \ref{thm:AsyJac} are sufficient for deriving the pointwise error estimates of Jacobi spectral differentiation. 
\end{remark}

\subsection{Pointwise error estimates of Jacobi spectral differentiations}
In this subsection we consider pointwise error estimates of spectral differentiations using Jacobi projection for the function \eqref{def:ModelFun}. Let $\mathbb{N}_0$ denote the set of all nonnegative integers. We define the remainder of Jacobi spectral differentiation of order $m$ by
\begin{equation}\label{def:Remainder}
\mathcal{R}_{n}^{m}(x) = \frac{\mathrm{d}^m}{\mathrm{d}x^m} \big( f(x) - S_n^{(\alpha,\beta)}(x) \big),
\end{equation}
where $m\in\mathbb{N}_0$. Note that $\mathcal{R}_{n}^{m}(x)$ denotes the remainder of Jacobi approximation when $m=0$. Moreover, it also depends on the parameters $\alpha,\beta,\sigma,\xi$, but we omit the dependence here for notational simplicity.


We define the following two functions
\begin{align}\label{def:PsiFun}
\Psi_{\nu}^{\mathrm{C}}(x,n) = \sum_{k=n+1}^{\infty}
\frac{\cos(kx)}{k^{\nu+1}}, \qquad \Psi_{\nu}^{\mathrm{S}}(x,n) =
\sum_{k=n+1}^{\infty} \frac{\sin(kx)}{k^{\nu+1}},
\end{align}
where $n\in\mathbb{N}_0$, $x\in\mathbb{R}$ and $\nu\in\mathbb{R}$ when
$x(\mathrm{mod}~2\pi)\neq0$ and $\nu>0$ when
$x(\mathrm{mod}~2\pi)=0$. Both functions were introduced in \cite{Wang2023a} to analyze the pointwise error estimates of Chebyshev approximations for the function
\eqref{def:ModelFun}. Their asymptotic behaviors as $n\rightarrow\infty$ will be useful.
\begin{lemma}\label{lem:AsmPsiFun}
If $x(\mathrm{mod}~2\pi)\neq0$, we have
\begin{align}
\Psi_{\nu}^{\mathrm{C}}(x,n) &= \frac{\sin((n+1/2)x)}{-2\sin(x/2)}
n^{-\nu-1} + O(n^{-\nu-2}),  \nonumber
\\[12pt]
\Psi_{\nu}^{\mathrm{S}}(x,n) &= \frac{\cos((n+1/2)x)}{2\sin(x/2)}
n^{-\nu-1} + O(n^{-\nu-2}).  \nonumber
\end{align}
If $x(\mathrm{mod}~2\pi)=0$, we have
\begin{align}
\Psi_{\nu}^{\mathrm{C}}(x,n) &= \frac{n^{-\nu}}{\nu} -
\frac{n^{-\nu-1}}{2} + O(n^{-\nu-2}), \quad
\Psi_{\nu}^{\mathrm{S}}(x,n) \equiv0. \nonumber
\end{align}
\end{lemma}
\begin{proof}
See \cite[Lemma~2.2]{Wang2023a}.
\end{proof}

Now we consider the pointwise error estimate of Jacobi spectral differentiations for the function \eqref{def:ModelFun}. To ensure the uniform convergence of the derivatives of the Jacobi expansion, we introduce the following assumption. For detailed discussion on the uniform convergence of a Jacobi series, we refer to \cite{Belenkii1989} .\\[2pt]
\textbf{Assumption} A. Assume that $\sigma>0$ is not an even integer when $\xi\in(-1,1)$ and is not an integer when $\xi=\pm1$. Moreover, we also assume that
\begin{align}
m &< \left\{
            \begin{array}{ll}
{\displaystyle \min\left\{\sigma, \frac{\sigma-\alpha+1/2}{2}, \frac{\sigma-\beta+1/2}{2} \right\}},   & \hbox{$\xi\in(-1,1)$,}    \\[10pt]
{\displaystyle \sigma + \min\left\{0, \frac{\beta-\alpha+1}{2} \right\} },   & \hbox{$\xi=-1$,}     \\[10pt]
{\displaystyle \sigma + \min\left\{0, \frac{\alpha-\beta+1}{2} \right\} },   &
\hbox{$\xi=1$.}
            \end{array}
            \right. \nonumber
\end{align}
The main result of this work is stated in the following theorem.
\begin{theorem}\label{thm:JacobiDiff}
Let $m$ and $\sigma$ satisfy the assumption $\mathrm{A}$. The remainder of Jacobi spectral differentiation of order $m$ at each point $x\in\Omega$ satisfies
\begin{equation}
\mathcal{R}_{n}^{m}(x) = O(n^{-\kappa(x)}), \quad n\rightarrow\infty, \nonumber
\end{equation}
and the exponent $\kappa(x)$ is given below:
\begin{itemize}
\item[\rm(i)] If $\xi\in(-1,1)$, then for $x\neq\xi$,
\begin{align}
\kappa(x) &= \left\{
            \begin{array}{ll}
{\displaystyle \sigma+1/2-\beta-2m  },   & \hbox{$x=-1$,}   \\[5pt]
{\displaystyle \sigma+1/2-\alpha-2m },   & \hbox{$x=1$,}    \\[5pt]
{\displaystyle \sigma+1-m },             & \hbox{$x\in(-1,\xi)\cup(\xi,1)$,}
          \end{array}
          \right.
\end{align}
and for $x=\xi$,
\begin{align}
\kappa(x) &= \left\{
            \begin{array}{ll}
{\displaystyle \sigma+1-m},   & \mbox{$m$ odd,}    \\[5pt]
{\displaystyle \sigma-m  },   & \mbox{$m$ even.}
            \end{array}
            \right.
\end{align}

\item[\rm(ii)] If $\xi=-1$, then
\begin{align}
\kappa(x) &= \left\{
            \begin{array}{ll}
{\displaystyle 2\sigma - 2m },            & \hbox{$x=\xi$,}       \\[5pt]
{\displaystyle 2\sigma+\beta-\alpha+1-2m},   & \hbox{$x=-\xi$,}   \\[5pt]
{\displaystyle 2\sigma+\beta+3/2-m },     & \hbox{$x\in(-1,1)$.}
            \end{array}
            \right.
\end{align}
If $\xi=1$, then
\begin{align}
\kappa(x) &= \left\{
            \begin{array}{ll}
{\displaystyle 2\sigma - 2m },            & \hbox{$x=\xi$,}   \\[5pt]
{\displaystyle 2\sigma+\alpha-\beta+1-2m},   & \hbox{$x=-\xi$,}   \\[5pt]
{\displaystyle 2\sigma+\alpha+3/2-m },     & \hbox{$x\in(-1,1)$.}
            \end{array}
            \right.
\end{align}
\end{itemize}
\end{theorem}
\begin{proof}
We first consider the proof of (i). By the derivative formula of Jacobi polynomials in \cite[Equation~(18.9.15)]{Olver2010} and the asymptotic expansion of the ratio of gamma functions in \cite[Equation~(5.11.13)]{Olver2010}, we have
\begin{align}\label{eq:PointErrorS1}
\mathcal{R}_{n}^{m}(x) &= \sum_{k=n+1}^{\infty} a_k^{(\alpha,\beta)}
\frac{(k+\alpha+\beta+1)_m}{2^m} P_{k-m}^{(\alpha+m,\beta+m)}(x) \nonumber \\
&= \frac{1}{2^m} \sum_{k=n+1}^{\infty} a_k^{(\alpha,\beta)}
\frac{\Gamma(k+\alpha+\beta+m+1)}{\Gamma(k+\alpha+\beta+1)} P_{k-m}^{(\alpha+m,\beta+m)}(x) \nonumber \\
&= \frac{\mathcal{A}_{\sigma,\xi}^{\alpha,\beta}}{2^m} \sum_{k=n+1}^{\infty} \frac{\cos(k\arccos(\xi)-\psi_{\alpha,\beta}(\xi))}{k^{\sigma+1/2-m}} P_{k-m}^{(\alpha+m,\beta+m)}(x) + \mathrm{HOT}, \nonumber
\end{align}
where we have used the asymptotic of Jacobi coefficients in Theorem \ref{thm:AsyJac} in the last step and $\mathrm{HOT}$ denotes higher order term. Next we consider the asymptotic of $\mathcal{R}_{n}^{m}(x)$ at each point $x\in\Omega$ as $n\rightarrow\infty$. We first consider $x=-1$. Since
\begin{equation}
P_k^{(\alpha,\beta)}(-1) = (-1)^k \frac{(\beta+1)_k}{k!} = \frac{(-1)^k k^{\beta}}{\Gamma(\beta+1)}\left( 1 + O\left(\frac{1}{k}\right) \right), \quad k\gg1, \nonumber
\end{equation}
we obtain
\begin{align}
\mathcal{R}_{n}^{m}(-1) &=
\frac{(-1)^m \mathcal{A}_{\sigma,\xi}^{\alpha,\beta}}{2^{m} \Gamma(m+\beta+1)} \sum_{k=n+1}^{\infty}
\frac{(-1)^k \cos(k\arccos(\xi)-\psi_{\alpha,\beta}(\xi))}{k^{\sigma+1/2-2m-\beta}} + \mathrm{HOT}  \nonumber \\
&= \frac{(-1)^m \mathcal{A}_{\sigma,\xi}^{\alpha,\beta}}{2^{m} \Gamma(m+\beta+1)} \left[ \cos(\psi_{\alpha,\beta}(\xi)) \sum_{k=n+1}^{\infty}
\frac{(-1)^k \cos(k\arccos(\xi))}{k^{\sigma+1/2-2m-\beta}} \right. \nonumber \\
& \left. + \sin(\psi_{\alpha,\beta}(\xi)) \sum_{k=n+1}^{\infty}
\frac{(-1)^k \sin(k\arccos(\xi))}{k^{\sigma+1/2-2m-\beta}} \right] + \mathrm{HOT} \nonumber \\
&= \frac{(-1)^m \mathcal{A}_{\sigma,\xi}^{\alpha,\beta}}{2^{m} \Gamma(m+\beta+1)} \left[ \frac{\cos(\psi_{\alpha,\beta}(\xi))}{2} \left( \Psi_{\nu}^{\mathrm{C}}(2\varphi_{\xi}^{-}(-1),n) + \Psi_{\nu}^{\mathrm{C}}(2\varphi_{\xi}^{+}(-1),n) \right) \right. \nonumber \\
& \left. + \frac{\sin(\psi_{\alpha,\beta}(\xi))}{2} \left( \Psi_{\nu}^{\mathrm{S}}(2\varphi_{\xi}^{+}(-1),n) - \Psi_{\nu}^{\mathrm{S}}(2\varphi_{\xi}^{-}(-1),n) \right) \right] + \mathrm{HOT}, \nonumber
\end{align}
where $\nu=\sigma-2m-\beta-1/2$ and $\varphi_{\xi}^{\pm}(x)=(\arccos(x)\pm\arccos(\xi))/2$. Since $\varphi_{\xi}^{+}(-1)\in(\pi/2,\pi)$ and $\varphi_{\xi}^{-}(-1)\in(0,\pi/2)$, by Lemma \ref{lem:AsmPsiFun} and after some calculations, we obtain that
\begin{align}
\mathcal{R}_{n}^{m}(-1) = \frac{(-1)^{n+m+1} \mathcal{A}_{\sigma,\xi}^{\alpha,\beta} \cos(\psi_{\alpha,\beta}(\xi)-(n+1/2)\arccos(\xi))}{2^{m+1}\Gamma(m+\beta+1)\cos(\arccos(\xi)/2)} n^{2m+\beta-\sigma-1/2} + \mathrm{HOT},  \nonumber
\end{align}
and this proves the estimate at $x=-1$. Similarly, we can deduce that
\begin{align}
\mathcal{R}_{n}^{m}(1) = \frac{\mathcal{A}_{\sigma,\xi}^{\alpha,\beta} \sin(\psi_{\alpha,\beta}(\xi)-(n+1/2)\arccos(\xi))}{2^{m+1}\Gamma(m+\alpha+1)\sin(\arccos(\xi)/2)} n^{2m+\alpha-\sigma-1/2} + \mathrm{HOT},  \nonumber
\end{align}
and this proves the estimate at $x=1$. We next consider $x\in(-1,1)$. In this case, by \cite[Theorem~8.21.8]{Szego1975} we know that
\begin{align}
P_{k-m}^{(\alpha+m,\beta+m)}(x) &= \frac{2^{m+(\alpha+\beta+1)/2} \cos\left( k\arccos(x) - \psi_{\alpha,\beta}(x) - m\pi/2 \right) }{\sqrt{k\pi\omega_{m+\alpha+1/2,m+\beta+1/2}(x)}} + O\left(\frac{1}{k^{3/2}}\right), \nonumber
\end{align}
where the error term holds uniformly on any compact subsets of $(-1,1)$, and therefore
\begin{align}
\mathcal{R}_{n}^{m}(x) &= \frac{2^{(\alpha+\beta+1)/2} \mathcal{A}_{\sigma,\xi}^{\alpha,\beta} }{\sqrt{\pi \omega_{m+\alpha+1/2,m+\beta+1/2}(x) }} \nonumber \\
&~ \times \sum_{k=n+1}^{\infty}
\frac{\cos(k\arccos(\xi)-\psi_{\alpha,\beta}(\xi)) \cos(k\arccos(x) - \psi_{\alpha,\beta}(x) - m\pi/2)}{k^{\sigma+1-m}} + \mathrm{HOT}. \nonumber
\end{align}
Let $\Upsilon$ denote the sum on the right-hand side of the above equation. By direct calculations, $\Upsilon$ can be written as
\begin{align}
\Upsilon 
&= \frac{1}{2}\cos\left(\psi_{\alpha,\beta}(\xi)-\psi_{\alpha,\beta}(x)-m\pi/2\right) \Psi_{\sigma-m}^{\mathrm{C}}(2\varphi_{\xi}^{-}(x),n) \nonumber \\
& - \frac{1}{2}\sin\left(\psi_{\alpha,\beta}(\xi)-\psi_{\alpha,\beta}(x)-m\pi/2\right) \Psi_{\sigma-m}^{\mathrm{S}}(2\varphi_{\xi}^{-}(x),n) \nonumber \\
& + \frac{1}{2}\cos\left(\psi_{\alpha,\beta}(\xi)+\psi_{\alpha,\beta}(x)+m\pi/2\right) \Psi_{\sigma-m}^{\mathrm{C}}(2\varphi_{\xi}^{+}(x),n) \nonumber \\
& + \frac{1}{2}\sin\left(\psi_{\alpha,\beta}(\xi)+\psi_{\alpha,\beta}(x)+m\pi/2\right) \Psi_{\sigma-m}^{\mathrm{S}}(2\varphi_{\xi}^{+}(x),n). \nonumber
\end{align}
Next, we divide the discussion into two cases: $x\in(-1,\xi)\cup(\xi,1)$ and $x=\xi$. For $x\in(-1,\xi)\cup(\xi,1)$, we see that $\varphi_{\xi}^{+}(x)\in(0,\pi)$ and $\varphi_{\xi}^{-}(x)\in(-\pi/2,0)\cup(0,\pi/2)$. By Lemma \ref{lem:AsmPsiFun} and after some calculations, we deduce that
\begin{align}
\Upsilon = \frac{1}{4}& \left( \frac{\sin(\psi_{\alpha,\beta}(\xi)+\psi_{\alpha,\beta}(x)+m\pi/2-(2n+1)\varphi_{\xi}^{+}(x))}{\sin(\varphi_{\xi}^{+}(x))} \right. \nonumber \\
& \left. - \frac{\sin(\psi_{\alpha,\beta}(\xi)-\psi_{\alpha,\beta}(x)-m\pi/2+(2n+1)\varphi_{\xi}^{-}(x))}{\sin(\varphi_{\xi}^{-}(x))}   \right) n^{m-\sigma-1} + \mathrm{HOT},  \nonumber
\end{align}
and thus
\begin{align}
\mathcal{R}_{n}^{m}(x) &= \frac{2^{(\alpha+\beta-3)/2} \mathcal{A}_{\sigma,\xi}^{\alpha,\beta} }{\sqrt{\pi \omega_{m+\alpha+1/2,m+\beta+1/2}(x) }} \left( \frac{\sin(\psi_{\alpha,\beta}(\xi)+\psi_{\alpha,\beta}(x)+m\pi/2-(2n+1)\varphi_{\xi}^{+}(x))}{\sin(\varphi_{\xi}^{+}(x))} \right. \nonumber \\
& \left. - \frac{\sin(\psi_{\alpha,\beta}(\xi)-\psi_{\alpha,\beta}(x)-m\pi/2+(2n+1)\varphi_{\xi}^{-}(x))}{\sin(\varphi_{\xi}^{-}(x))}   \right) n^{m-\sigma-1} + \mathrm{HOT}.  \nonumber
\end{align}
This proves the estimate for $x\in(-1,\xi)\cup(\xi,1)$. 
For $x=\xi$, we find that $\varphi_{\xi}^{-}(\xi)=0$ and $\varphi_{\xi}^{+}(\xi)=\arccos(\xi)$, and thus
\begin{align}
\Upsilon = \frac{\cos\left(m\pi/2\right)}{2} \Psi_{\sigma-m}^{\mathrm{C}}(0,n) &+ \frac{\cos\left(2\psi_{\alpha,\beta}(\xi)+m\pi/2\right)}{2} \Psi_{\sigma-m}^{\mathrm{C}}(2\arccos(\xi),n) \nonumber \\
& + \frac{\sin\left(2\psi_{\alpha,\beta}(\xi)+m\pi/2\right)}{2} \Psi_{\sigma-m}^{\mathrm{S}}(2\arccos(\xi),n). \nonumber
\end{align}
When $m$ is odd, by Lemma \ref{lem:AsmPsiFun} we have
\begin{align}
\Upsilon &= \frac{\cos\left(2\psi_{\alpha,\beta}(\xi)+m\pi/2\right)}{2}  \Psi_{\sigma-m}^{\mathrm{C}}(2\arccos(\xi),n) \nonumber \\
&+ \frac{\sin\left(2\psi_{\alpha,\beta}(\xi)+m\pi/2\right)}{2} \Psi_{\sigma-m}^{\mathrm{S}}(2\arccos(\xi),n) \nonumber \\
&= \frac{\sin(2\psi_{\alpha,\beta}(\xi)+m\pi/2-(2n+1)\arccos(\xi))}{4\sin(\arccos(\xi))} n^{m-\sigma-1} + \mathrm{HOT}, \nonumber
\end{align}
and thus
\begin{align}
\mathcal{R}_{n}^{m}(\xi) &= \frac{2^{(\alpha+\beta-3)/2} \mathcal{A}_{\sigma,\xi}^{\alpha,\beta} \sin(2\psi_{\alpha,\beta}(\xi)+m\pi/2-(2n+1)\arccos(\xi)) }{\sin(\arccos(\xi)) \sqrt{\pi \omega_{m+\alpha+1/2,m+\beta+1/2}(\xi) }} n^{m-\sigma-1} + \mathrm{HOT}. \nonumber
\end{align}
When $m$ is even, by Lemma \ref{lem:AsmPsiFun} again we deduce that
\[
\Upsilon = \frac{(-1)^{m/2}}{2(\sigma-m)} n^{m-\sigma} + \mathrm{HOT},
\]
and thus
\begin{align}
\mathcal{R}_{n}^{m}(\xi) &= \frac{(-1)^{m/2} 2^{(\alpha+\beta-1)/2} \mathcal{A}_{\sigma,\xi}^{\alpha,\beta} }{(\sigma-m) \sqrt{\pi \omega_{m+\alpha+1/2,m+\beta+1/2}(\xi) }} n^{m-\sigma} + \mathrm{HOT}. \nonumber
\end{align}
This proves the estimate for $x=\xi$ and the proof of (i) is complete.

As for (ii), i.e., $\xi=\pm1$, the estimates of
$\mathcal{R}_{n}^{m}(x)$ can similarly be derived. We omit the details and this ends the proof.
\end{proof}


As a direct consequence, we obtain the convergence rate of Jacobi spectral differentiations in the maximum norm.
\begin{corollary}\label{cor:JacMax}
As $n\rightarrow\infty$, the following statements hold.
\begin{itemize}
\item[\rm(i)] If $\xi\in(-1,1)$, then
\begin{equation}
\|\mathcal{R}_{n}^{m}\|_{\infty} = \left\{
            \begin{array}{ll}
{\displaystyle O(n^{-\sigma+\max\{m,~2m+\alpha-1/2,~2m+\beta-1/2 \}})},   & \hbox{$m$ even,}    \\[5pt]
{\displaystyle O(n^{-\sigma+\max\{m-1,~2m+\alpha-1/2,~2m+\beta-1/2 \}})},   & \hbox{$m$ odd.}
            \end{array}
            \right.
\end{equation}
When $m=0$, the maximum error is attained at the singularity $x=\xi$ if $\max\{\alpha,\beta\}<1/2$ and at $x=-1$ if $\beta>\max\{\alpha,1/2\}$ and at $x=1$ if $\alpha>\max\{\beta,1/2\}$. When $m\geq1$, however, the maximum error is always attained at one of the endpoints. More precisely, the maximum error is attained at $x=-1$ when $\beta>\alpha$ and at $x=1$ when $\alpha>\beta$.

\item[\rm(ii)] If $\xi=-1$, then
\begin{equation}
\|\mathcal{R}_{n}^{m}\|_{\infty} = O(n^{-2\sigma+2m+\max\{0,\alpha-\beta-1\}}),
\end{equation}
and the maximum error is always attained at one of the endpoints. More precisely, the maximum error is attained at $x=\xi$ if $\alpha<\beta+1$ and at $x=-\xi$ if $\alpha>\beta+1$. If $\xi=1$, then
\begin{equation}
\|\mathcal{R}_{n}^{m}\|_{\infty} = O(n^{-2\sigma+2m+\max\{0,\beta-\alpha-1\}}),
\end{equation}
and the maximum error is always attained at one of the endpoints. More precisely, the maximum error is attained at $x=\xi$ if $\beta<\alpha+1$ and at $x=-\xi$ if $\beta>\alpha+1$.
\end{itemize}
\end{corollary}
\begin{proof}
It follows from Theorem \ref{thm:JacobiDiff} immediately.
\end{proof}

Several remarks on Theorem \ref{thm:JacobiDiff} and the above corollary are in order.
\begin{remark}
When the order of Jacobi spectral differentiation increases from $m$ to $m+1$, we see that the convergence rate deteriorates two orders at both endpoints and only one order at each point in the smooth region. The pattern of deterioration at the singularity $x=\xi$ is quite different. More specifically, the convergence rate still deteriorates two orders if $\xi=\pm1$ and does not deteriorate when $m$ is even and deteriorates two order when $m$ is odd if $\xi\in(-1,1)$.
\end{remark}

\begin{remark}
When $m=0$, the convergence rate of Jacobi approximation at each point in the smooth region, i.e., $x\in(-1,\xi)\cup(\xi,1)$, is always faster than at the singularity $x=\xi$, regardless of $\xi\in(-1,1)$ or $\xi=\pm1$. This justifies the error localization property of Jacobi approximations, that is, their maximum error is always attained at one of the critical points (i.e., singularities and endpoints). Note that the estimates of $\mathcal{R}_{n}^{m}(x)$ for $m=0$ have been derived in \cite{Xiang2023} by using the reproducing kernel of Jacobi polynomials. Here we derive the estimates in a more general setting. Moreover, for $m\geq1$, the maximum error of Jacobi spectral differentiation of order $m$ is always dominated by the errors at the endpoints for all $\xi\in\Omega$.
\end{remark}

\begin{remark}\label{rk:Bernstein}
Bernstein in \cite{Bern1913,Bern1938} initiated the study of the limit
\[
\mu(\sigma) := \lim_{n\rightarrow\infty} n^{\sigma} \|f -
p_n^{*} \|_{\infty},
\]
for the function $f(x)=|x|^{\sigma}$ and $p_n^{*}$ denotes the best approximation of degree $n$ to $f$ in the maximum norm. The problem of finding $\mu(\sigma)$ has attracted considerable attention in approximation theory community. Here we consider an analogue for $S_n^{(\alpha,\beta)}(x)$. When $\max\{\alpha,\beta\}<1/2$, from Corollary \ref{cor:JacMax} we know that the maximum error of $S_n^{(\alpha,\beta)}(x)$ is attained at $\xi=0$ for large $n$. In this case, by the proof of Theorem \ref{thm:JacobiDiff} we know that
\[
\mathcal{R}_{n}^{m}(\xi) = \frac{(-1)^{m/2} 2^{(\alpha+\beta-1)/2} \mathcal{A}_{\sigma,\xi}^{\alpha,\beta} }{(\sigma-m) \sqrt{\pi \omega_{m+\alpha+1/2,m+\beta+1/2}(\xi) }} n^{m-\sigma} + \mathrm{HOT},
\]
and thus for $m=0$ and $\xi=0$,
\begin{align}
\lim_{n\rightarrow\infty} n^{\sigma} \|f-S_n^{(\alpha,\beta)}\|_{\infty}
&= \frac{2^{(\alpha+\beta-1)/2} |\mathcal{A}_{\sigma,0}^{\alpha,\beta}| }{\sigma\sqrt{\pi}} =  \frac{2\Gamma(\sigma)}{\pi} \left|\sin\left(\frac{\sigma\pi}{2}\right) \right|. \nonumber
\end{align}
Note that the last constant has been derived in \cite[Equation~(2.15)]{Wang2023a} in the case of Chebyshev approximation (i.e., $\alpha=\beta=-1/2$).
\end{remark}

\begin{figure}[htbp]
\centering
\includegraphics[width=0.6\textwidth,height=0.45\textwidth]{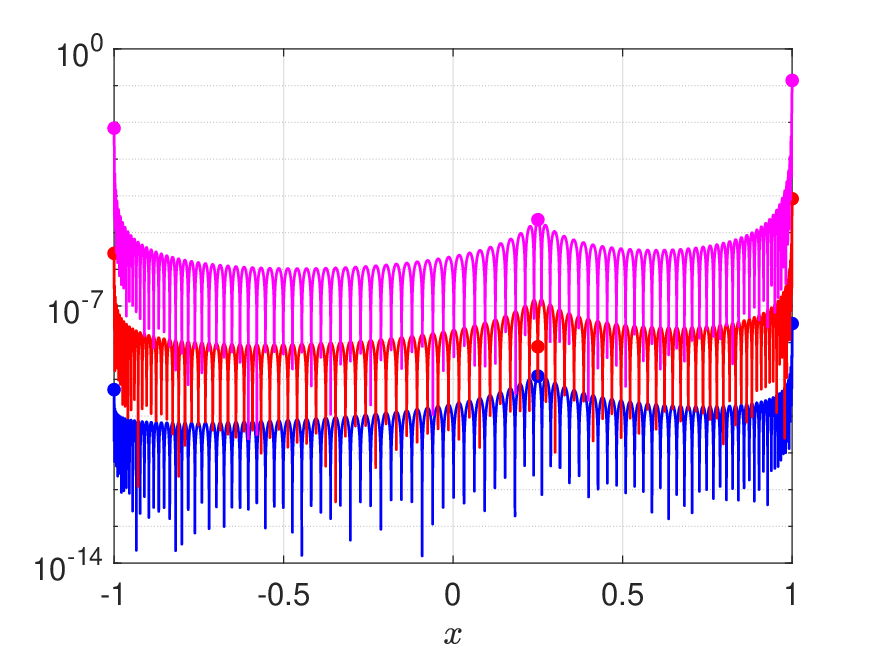}
\caption{Pointwise errors of the zero- and first- and second-order Jacobi differentiations (from bottom to top). Here $n=100$ and the points indicate the errors at the critical points.}\label{fig:JacobiDiff1}
\end{figure}

\begin{example}
Consider the pointwise error of Jacobi approximation and differentiations for $f(x)=|x-1/4|^5$. It is clear that $\xi=1/4,\sigma=5$ and, by Theorem \ref{thm:JacobiDiff}, the error decays at the rate $O(n^{2m+\beta-11/2})$ at $x=-1$ and $O(n^{2m+\alpha-11/2})$ at $x=1$ and $O(n^{m-6})$ at $x\in(-1,0)\cup(0,1)$. At the point $x=0$, the error decays at the rate $O(n^{m-5})$ when $m=0,2,4,\ldots$ and $O(n^{m-6})$ when $m=1,3,5,\ldots$. In Figure \ref{fig:JacobiDiff1} we plot the pointwise errors of Jacobi approximation and differentiations with $\alpha=1,\beta=0$ for $m=0,1,2$. Clearly, we see that the maximum errors are always dominated by the errors at one of the critical points (endpoints and singularity). To verify the predicted convergence rates of Jacobi differentiations at each point $x\in[-1,1]$, we plot in Figure \ref{fig:JacobiDiff2} the first and second order Jacobi differentiations at five points $x=-1,-1/4,0,1/4,1$. The convergence rates are $O(n^{-5/2})$ at $x=1$ and $O(n^{-7/2})$ at $x=-1$ and $O(n^{-5})$ at $x=0,\pm1/4$ for the first order Jacobi differentiation, and $O(n^{-1/2})$ at $x=1$ and $O(n^{-3/2})$ at $x=-1$ and $O(n^{-3})$ at $x=1/4$ and $O(n^{-4})$ at $x=-1/4,0$ for the second order Jacobi differentiation. We see from Figure \ref{fig:JacobiDiff2} that the convergence rates at these points are consistent with our predicted rates.
\end{example}

\begin{figure}[htbp]
\centering
\includegraphics[width=0.49\textwidth,height=0.4\textwidth]{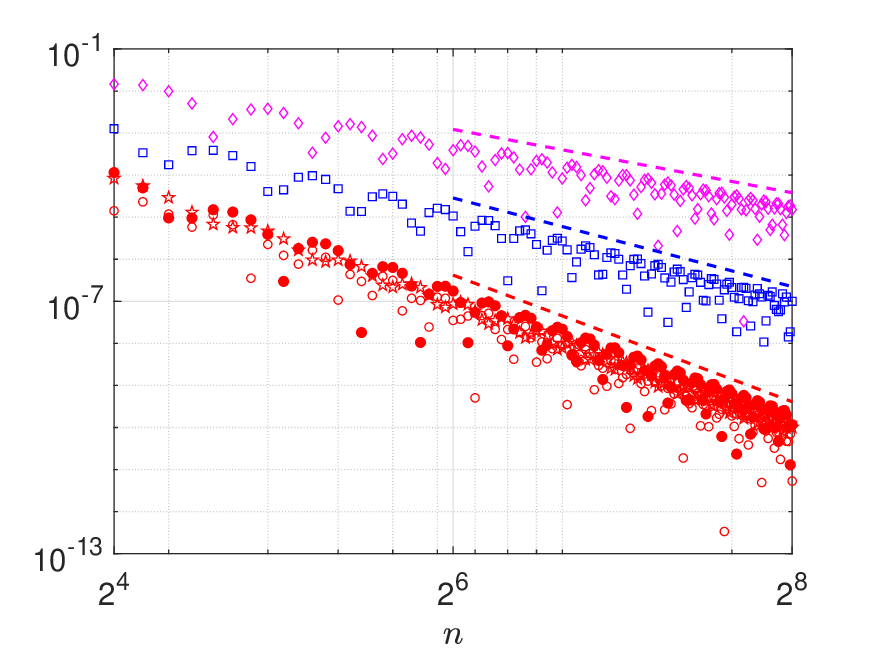}
\includegraphics[width=0.49\textwidth,height=0.4\textwidth]{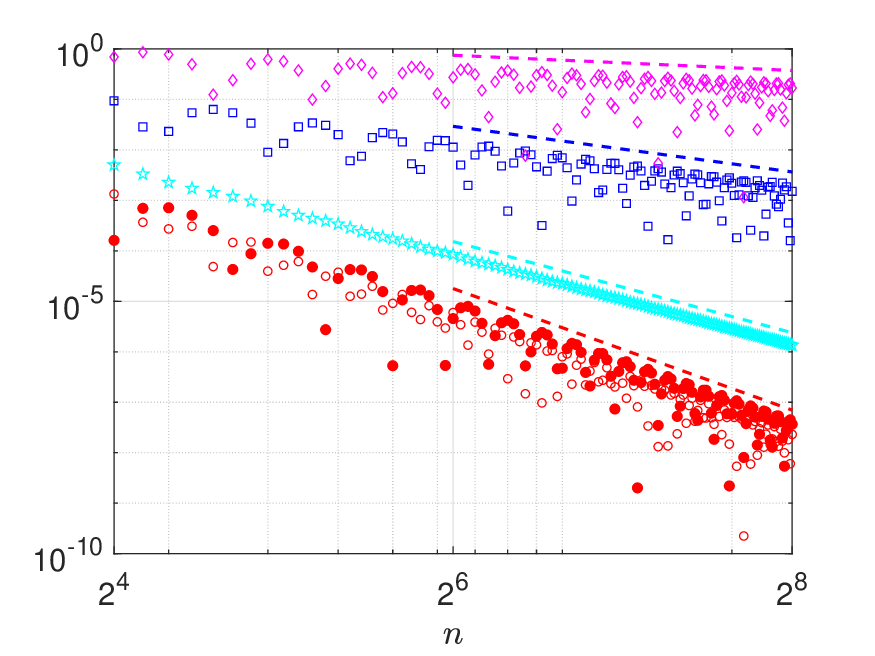}
\caption{Errors of the first- (left) and second-order (right) Jacobi differentiations as a function of $n$ at $x=1$ ($\Diamond$), $x=-1$ ($\Box$) and $x=-1/4$ ($\circ$), $x=0$ ($\bullet$) and $x=1/4$ (\ding{73}). The dashed lines from top to bottom show the rates $O(n^{-s})$ with $s=5/2,7/2,5$ (left) and $s=1/2,3/2,3,4$ (right) .}\label{fig:JacobiDiff2}
\end{figure}

\section{Extensions and discussions}\label{sec:Extension}
In this section, we discuss several topics that are closely related to Jacobi spectral approximation and differentiation.

\subsection{Spectral differentiation using Chebyshev interpolation}\label{sec:ChebInterp}
Chebyshev interpolation plays an important role in many applications, such as Clenshaw-Curtis quadrature, rootfinding and Chebyshev spectral collocation methods for solving differential and integral equations (see, e.g., \cite{Trefethen2000}). Let $\{x_j\}_{j=0}^{n}$ be the set of Chebyshev--Lobatto points (also known as Clenshaw--Curtis points), i.e., $ x_j=\cos(j\pi/n)$, and let $p_n(x)$ be the polynomial of degree $n$ which interpolates $f(x)$ at these points. It is known that
\begin{align}
p_n(x) = \sum_{k=0}^{n}{''} c_k T_k(x), \quad  c_k =
\frac{2}{n} \sum_{j=0}^{n}{''} f(x_j) T_k(x_j),  \nonumber
\end{align}
where the double prime indicates that both the first and last terms
of the summation are to be halved and $T_k(x)$ is the Chebyshev polynomial of the first kind of degree $k$. It is well known that the coefficients
$\{c_k\}_{k=0}^{n}$ can be computed rapidly by using the FFT in
only $O(n\log n)$ operations. Below we state pointwise error estimates of Chebyshev spectral differentiation using $p_n(x)$.
\begin{theorem}\label{thm:Cheb}
Let $f$ be defined in \eqref{def:ModelFun} and assume that $\sigma>0$ is not an even integer when $\xi\in(-1,1)$ and is not an integer when $\xi=\pm1$. As $n\rightarrow\infty$, the following results hold.
\begin{itemize}
\item For $\xi\in(-1,1)$ and for $m<\min\{\sigma,1+\sigma/2\}$, then
\begin{align}\label{eq:ChebInterp1}
f^{(m)}(x) - p_n^{(m)}(x) = \left\{
            \begin{array}{ll}
{\displaystyle O(n^{m-\sigma})},      & \hbox{$x=\xi$,}                   \\[5pt]
{\displaystyle O(n^{m-1-\sigma}) },   & \hbox{$x\in(-1,\xi)\cup(\xi,1)$,} \\[5pt]
{\displaystyle O(n^{2m-2-\sigma}) },  & \hbox{$x=\pm1$,}
            \end{array}
            \right.
\end{align}
and the last result holds for $m\geq1$ since it is zero when $m=0$.

\item For $\xi=\pm1$ and for $m<\sigma$, then
\begin{align}\label{eq:ChebInterp2}
f^{(m)}(x) - p_n^{(m)}(x) = \left\{
            \begin{array}{ll}
{\displaystyle O(n^{m-1-2\sigma}) },   & \hbox{$x\in(-1,1)$,} \\[5pt]
{\displaystyle O(n^{2m-2\sigma})},     & \hbox{$x=\xi$,}      \\[5pt]
{\displaystyle O(n^{2m-2-2\sigma})},   & \hbox{$x=-\xi$.}
            \end{array}
            \right.
\end{align}
and the last two results hold for $m\geq1$ since they are zero when $m=0$.
\end{itemize}
\end{theorem}
\begin{proof}
Since $f$ has the following uniformly convergent Chebyshev series
\begin{align}
f(x) = \sum_{k=0}^{\infty}{'} a_k T_k(x), \quad a_k = \frac{2}{\pi}
\int_{-1}^{1} \frac{f(x) T_k(x)}{\sqrt{1-x^2}} \mathrm{d}x, \nonumber
\end{align}
where the prime indicates that the first term of the summation should be
halved. By virtue of the aliasing formula of the coefficients $\{a_k\}$ and $\{c_k\}$, we know from \cite[Equation~(3.16)]{Wang2023a} that
\begin{align}
(f - p_n)(x) &= \sum_{\ell=1}^{\infty} (1 - \cos(2\ell n\theta)) \sum_{k=(2\ell-1)n+1}^{(2\ell+1)n} a_k \cos(k\theta)   \nonumber \\
&~ - \sum_{\ell=1}^{\infty} \sin(2\ell n\theta)
\sum_{k=(2\ell-1)n+1}^{(2\ell+1)n} a_k \sin(k\theta),  \nonumber
\end{align}
where $x=\cos\theta$. Note that the asymptotics of the coefficients $\{a_k\}$ can be derived from Lemma \ref{thm:AsyJac} (see also \cite[Lemma 2.1]{Wang2023a}). The desired estimates \eqref{eq:ChebInterp1} and \eqref{eq:ChebInterp2} for $m=0$ can be derived immediately using Lemma \ref{lem:AsmPsiFun}. Further, differentiating the above equation with respect to $x$ yields
\begin{align}
(f - p_n){'}(x) &= \frac{1}{-\sin\theta} \left[ 2n\sin(2n\theta) \sum_{k=n+1}^{3n} a_k
\cos(k\theta) - (1 - \cos(2n\theta) ) \sum_{k=n+1}^{3n} a_k k \sin(k\theta)  \right. \nonumber \\
&\left. - 2n\cos(2n\theta) \sum_{k=n+1}^{3n} a_k \sin(k\theta) - \sin(2n\theta) \sum_{k=n+1}^{3n} a_k k \cos(k\theta)  - \cdots   \right]. \nonumber
\end{align}
When $x=1$, taking the limit $\theta\rightarrow0$ in the above equation gives
\begin{align}
(f - p_n){'}(1) &= -(2n)^2 \sum_{k=n+1}^{3n} a_k - (4n)^2 \sum_{k=3n+1}^{5n} a_k - \cdots \nonumber \\
& + 2\left( 2n \sum_{k=n+1}^{3n} a_k k + 4n \sum_{k=3n+1}^{5n} a_k k + \cdots \right). \nonumber
\end{align}
and when $x=-1$, taking the limit $\theta\rightarrow\pi$ gives
\begin{align}
(f - p_n){'}(-1) &= (2n)^2 \sum_{k=n+1}^{3n} a_k(-1)^k + (4n)^2 \sum_{k=3n+1}^{5n} a_k(-1)^k - \cdots \nonumber \\
& - 2\left( 2n \sum_{k=n+1}^{3n} a_k k (-1)^k + 4n \sum_{k=3n+1}^{5n} a_k k (-1)^k + \cdots \right). \nonumber
\end{align}
The estimates \eqref{eq:ChebInterp1} and \eqref{eq:ChebInterp2} for $m=1$ can be derived immediately by combining the above equations with Lemma \ref{lem:AsmPsiFun}. The estimates \eqref{eq:ChebInterp1} and \eqref{eq:ChebInterp2} for $m\geq2$ can be proved following the same line and we omit the details.
\end{proof}

As a direct consequence, we obtain the convergence rate of Chebyshev spectral differentiations using $p_n(x)$ in the maximum norm.
\begin{corollary}
As $n\rightarrow\infty$, the following results hold.
\begin{itemize}
\item For $\xi\in(-1,1)$, then
\[
\|f^{(m)} - p_n^{(m)}\|_{\infty} = O(n^{\max\{m,2m-2\}-\sigma}).
\]

\item For $\xi=\pm1$, then
\[
\|f^{(m)} - p_n^{(m)}\|_{\infty} = O(n^{\max\{2m,m-1\}-2\sigma}).
\]
\end{itemize}
\end{corollary}

\begin{remark}
The convergence rate of spectral differentiation using $p_n(x)$ also deteriorates two orders at both endpoints and only one order at each point in the smooth region. At the singularity $x=\xi$, the convergence rate deteriorates two orders if $\xi=\pm1$ and only one order if $\xi\in(-1,1)$. Moreover, comparing the convergence rates of spectral differentiation using Chebyshev interpolation and projection methods, we see that they converge at the same rate at the smooth region. At the critical points, i.e., endpoints and singularity, however, their convergence rates may have slight differences.
\end{remark}

In Figure \ref{fig:ChebDiff} we plot the errors of the first- and second-order differentiations of Chebyshev interpolant $p_n(x)$ at the points $x=1/3,2/3,1$ for the function $f(x)=|x-1/3|^{3}$. Clearly, this function corresponds to $\xi=1/3$ and $\sigma=3$ and therefore, by Theorem \ref{thm:Cheb}, the convergence rates of Chebyshev spectral differentiations of $p_n(x)$ are $O(n^{m-3})$ for $x=1/3$ and $O(n^{2m-5})$ for $x=\pm1$ and $O(n^{m-4})$ for $x\in(-1,1/3)\cup(1/3,1)$. We can see that the errors at the points $x=1/3,2/3,1$ decay at the predicted rates. In Figure \ref{fig:ChebDiff2} we plot the errors of the first- and second-order differentiations of Chebyshev interpolant $p_n(x)$ at the points $x=-1,1/5,1$ for the function $f(x)=(1+x)^{5/2}$ which has an endpoint singularity at $\xi=-1$. By Theorem \ref{thm:Cheb}, the convergence rates of Chebyshev spectral differentiations of $p_n(x)$ are $O(n^{2m-5})$ for $x=-1$ and $O(n^{m-6})$ for $x=1/5$ and $O(n^{2m-7})$ for $x=1$. We can see that the errors at the points $x=-1,1/5,1$ decay at the predicted rates for $m=1,2$.

\begin{figure}[htbp]
\centering
\includegraphics[width=0.49\textwidth,height=0.4\textwidth]{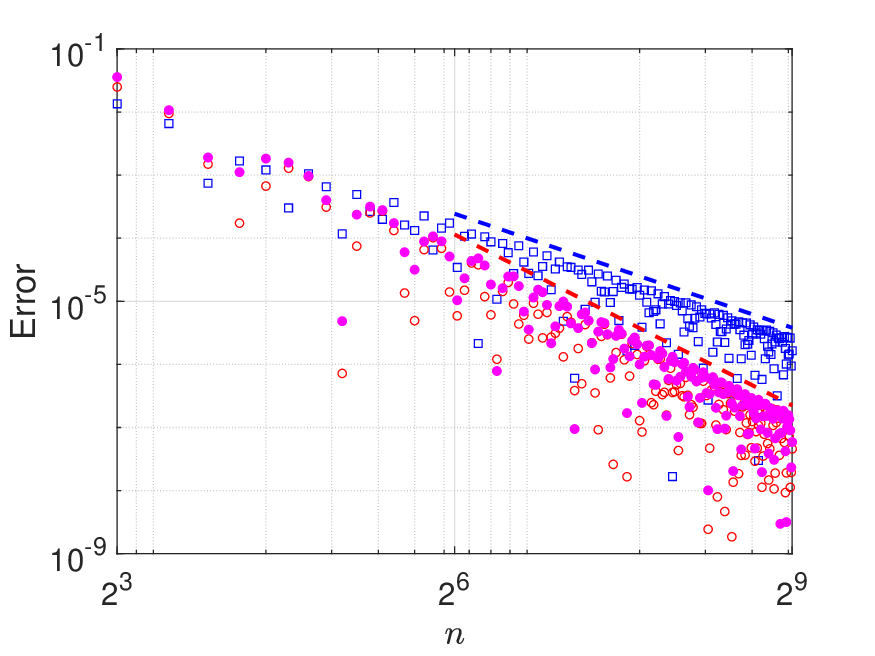}
\includegraphics[width=0.49\textwidth,height=0.4\textwidth]{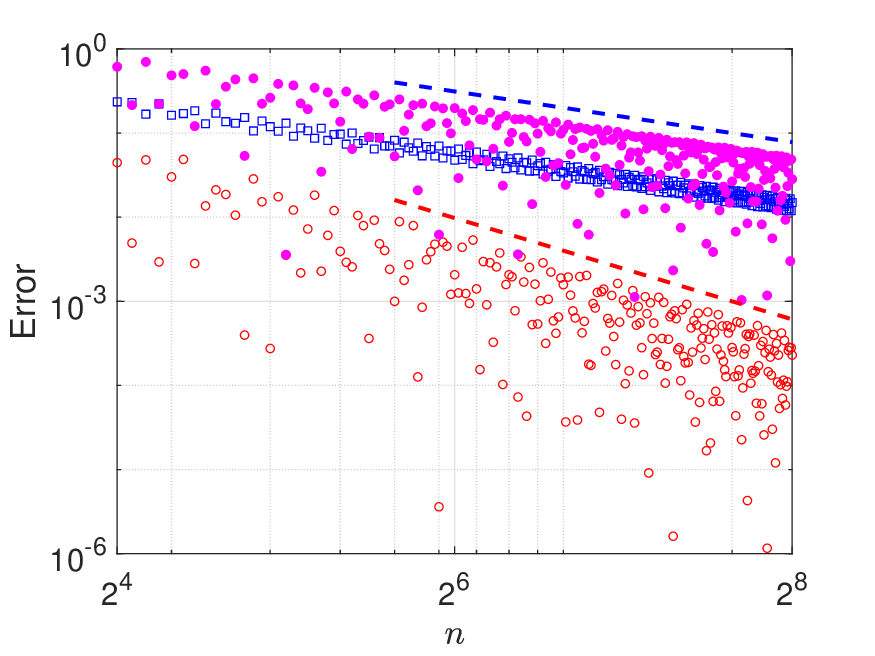}
\caption{Errors of the first- (left) and second-order (right) differentiations of Chebyshev interpolant $p_n(x)$ as a function of $n$ at $x=1/3$ ($\Box$), $x=2/3$ ($\circ$) and $x=1$ ($\bullet$) for $f(x)=|x-1/3|^{3}$. The dashed lines from top to bottom show the rates $O(n^{-s})$ with $s=2,3$ (left) and $s=1,2$ (right) .}\label{fig:ChebDiff}
\end{figure}

\begin{figure}[htbp]
\centering
\includegraphics[width=0.49\textwidth,height=0.4\textwidth]{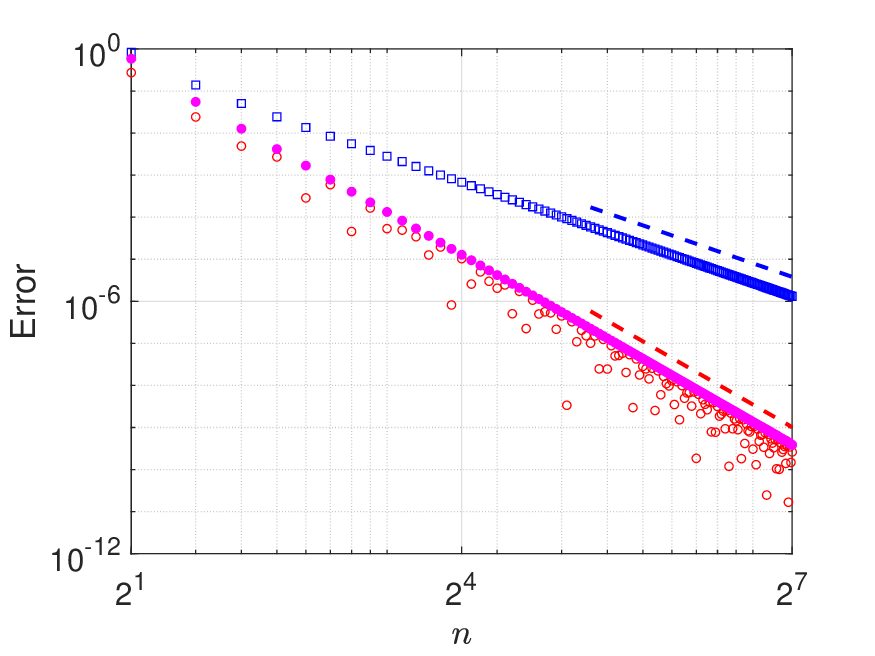}
\includegraphics[width=0.49\textwidth,height=0.4\textwidth]{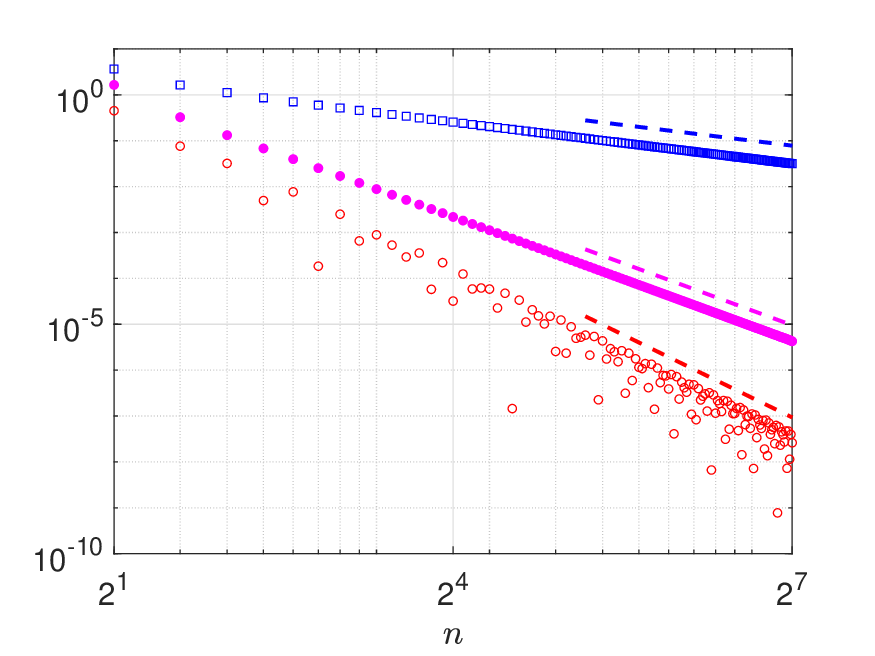}
\caption{Errors of the first- (left) and second-order (right) differentiations of Chebyshev interpolant $p_n(x)$ as a function of $n$ at $x=-1$ ($\Box$), $x=1/5$ ($\circ$) and $x=1$ ($\bullet$) for the function $f(x)=(1+x)^{5/2}$. The dashed lines from top to bottom show the rates $O(n^{-s})$ with $s=3,5$ (left) and $s=1,3,4$ (right).}\label{fig:ChebDiff2}
\end{figure}

\subsection{Is Lebesgue lemma sharp for Jacobi projection?}
Let $p_n^{*}$ denote the best polynomial approximation of degree $n$ to $f$ in the maximum norm. For $f\in{C}[-1,1]$, the classical Lebesgue lemma states that
\begin{equation}\label{eq:Lebesgue}
\|f - S_n^{(\alpha,\beta)}\|_\infty \leq (1 + \Lambda_n^{(\alpha,\beta)}) \|f - p_n^{*}\|_\infty,
\end{equation}
and $\Lambda_n^{(\alpha,\beta)}$ is the Lebesgue constant of $S_n^{(\alpha,\beta)}$. From \cite{Frenzen1986,Rau1929} we know that
\begin{equation}
\Lambda_n^{(\alpha,\beta)} = \left\{
            \begin{array}{ll}
{\displaystyle O(n^{\max\{\alpha,\beta\}+1/2})},   & \hbox{$\max\{\alpha,\beta\}>-1/2$,}    \\[5pt]
{\displaystyle O(\log n)},   & \hbox{$\beta\leq\alpha=-1/2$ or $\alpha\leq\beta=-1/2$,}  \\[5pt]
{\displaystyle O(1)},   & \hbox{$\max\{\alpha,\beta\}<-1/2$,}
            \end{array}
            \right.
\end{equation}
as $n\rightarrow\infty$. Is Lebesgue lemma sharp for the error estimate of $S_n^{(\alpha,\beta)}$? Here we consider the model function \eqref{def:ModelFun} to gain some insight. It is known that $\|f - p_n^{*}\|_\infty=O(n^{-\sigma})$ when $\xi\in(-1,1)$ and $\|f - p_n^{*}\|_\infty=O(n^{-2\sigma})$ when $\xi=\pm1$. On the other hand, by Corollary \ref{cor:JacMax}, we know that the actual convergence rate of $S_n^{(\alpha,\beta)}$ is
\begin{equation}\label{eq:Optimal}
\|f - S_n^{(\alpha,\beta)}\|_\infty = \left\{
            \begin{array}{ll}
{\displaystyle O(n^{-\sigma+\max\{0,\alpha-1/2,\beta-1/2 \}})},   & \hbox{$\xi\in(-1,1)$,}    \\[5pt]
{\displaystyle O(n^{-2\sigma+\max\{0,\alpha-\beta-1 \}})},   & \hbox{$\xi=-1$,}  \\[5pt]
{\displaystyle O(n^{-2\sigma+\max\{0,\beta-\alpha-1\}}) },   & \hbox{$\xi=1$.}
            \end{array}
            \right.
\end{equation}
We see that the convergence rates predicted by Lebesgue lemma might be misleading in certain situations, especially when $\alpha,\beta$ are close and both are large. To illustrate this, consider the function $f(x)=(1-x)^{7/2}$, i.e., $\xi=1$ and $\sigma=7/2$. When $\alpha=4$ and $\beta=5$, the predicted convergence rate by Lebesgue lemma in \eqref{eq:Lebesgue} is only $O(n^{-3/2})$. However, by \eqref{eq:Optimal} the actual convergence rate of $S_n^{(\alpha,\beta)}$ is $O(n^{-7})$. We see that Lebesgue lemma gives a rather bad prediction.

\subsection{Superconvergence of Jacobi approximation and differentiation}\label{sec:super}
Superconvergence theory has received much attention in finite element methods for differential equations and collocation methods for Volterra integral equations (see, e.g.,
\cite{Brunner2004,Wahlbin1995}). In the case of spectral methods, however, only a few studies had been conducted in the literature (see, e.g., \cite{Zhang2012}).

We consider the superconvergence points of Jacobi approximation and differentiation for the function \eqref{def:ModelFun}. In the case $\xi\in(-1,1)$, from the proof of Theorem \ref{thm:JacobiDiff}, we know for $x\in(-1,\xi)\cup(\xi,1)$ that
\begin{align}
\mathcal{R}_{n}^{m}(x) &= \frac{2^{(\alpha+\beta-3)/2} \mathcal{A}_{\sigma,\xi}^{\alpha,\beta} }{\sqrt{\pi \omega_{m+\alpha+1/2,m+\beta+1/2}(x) }} \left( \frac{\sin(\psi_{\alpha,\beta}(\xi)+\psi_{\alpha,\beta}(x)+m\pi/2-(2n+1)\varphi_{\xi}^{+}(x))}{\sin(\varphi_{\xi}^{+}(x))} \right. \nonumber \\
& \left. - \frac{\sin(\psi_{\alpha,\beta}(\xi)-\psi_{\alpha,\beta}(x)-m\pi/2+(2n+1)\varphi_{\xi}^{-}(x))}{\sin(\varphi_{\xi}^{-}(x))}   \right) n^{m-\sigma-1} + \mathrm{HOT}.  \nonumber
\end{align}
Thus, the superconvergence points can be derived by imposing the condition
\begin{align}
& \frac{\sin(\psi_{\alpha,\beta}(\xi)+\psi_{\alpha,\beta}(x)+m\pi/2-(2n+1)\varphi_{\xi}^{+}(x))}{\sin(\varphi_{\xi}^{+}(x))} \nonumber \\
& - \frac{\sin(\psi_{\alpha,\beta}(\xi)-\psi_{\alpha,\beta}(x)-m\pi/2+(2n+1)\varphi_{\xi}^{-}(x))}{\sin(\varphi_{\xi}^{-}(x))} = 0.   \nonumber
\end{align}
Clearly, the convergence rate of Jacobi approximation (i.e., $m=0$) and differentiation (i.e., $m\geq1$) at the roots of the above equation will be faster than the rate $O(n^{m-\sigma-1})$ when the roots are not close to the critical points $\{\pm1,\xi\}$. In the case $\xi=1$, by \eqref{eq:JacEndAsyR} and Lemma \ref{lem:AsmPsiFun}, we obtain for $x\in(-1,1)$ that
\begin{align}
\mathcal{R}_{n}^{m}(x) 
&= \frac{2^{(\alpha+\beta+1)/2-1} \mathcal{B}_{\sigma}^{R} }{\sqrt{\pi \omega_{m+\alpha+1/2,m+\beta+1/2}(x) }} \frac{\sin(\psi_{\alpha,\beta}(x)+m\pi/2-(n+1/2)\arccos(x))}{\sin(\arccos(x)/2) n^{2\sigma+\alpha+3/2-m}} + \mathrm{HOT}, \nonumber
\end{align}
and thus the superconvergence points can be derived by imposing the condition
\[
\sin(\psi_{\alpha,\beta}(x)+m\pi/2-(n+1/2)\arccos(x)) = 0,
\]
which gives
\[
x_j^R = \cos\left( \frac{(2\alpha+1)\pi/4 + m\pi/2 + j\pi}{n+1+(\alpha+\beta)/2}  \right), \quad  j\in\mathbb{Z}.
\]
In the case $\xi=-1$, we obtain for $x\in(-1,1)$ that
\begin{align}
\mathcal{R}_{n}^{m}(x) &= \frac{2^{(\alpha+\beta+1)/2-1} \mathcal{B}_{\sigma}^{L} }{\sqrt{\pi \omega_{m+\alpha+1/2,m+\beta+1/2}(x) }} \frac{\cos(\psi_{\alpha,\beta}(x)+m\pi/2-(n+1/2)\arccos(x))}{(-1)^{n+1}\cos(\arccos(x)/2)n^{2\sigma+\beta+3/2-m}} + \mathrm{HOT}, \nonumber
\end{align}
and thus the superconvergence points can be derived by imposing the condition
\[
\cos(\psi_{\alpha,\beta}(x)+m\pi/2-(n+1/2)\arccos(x)) = 0,
\]
which gives
\[
x_j^L = \cos\left( \frac{(2\alpha+1)\pi/4 + m\pi/2 + (j+1/2)\pi}{n+1+(\alpha+\beta)/2}  \right), \quad  j\in\mathbb{Z}.
\]
Note that these superconvergence points $\{x_j^R\}$ and $\{x_j^L\}$ are independent of  $\sigma$. Moreover, they are derived from the asymptotic expansion of $\mathcal{R}_{n}^{m}(x)$ for $x\in(-1,1)$, and thus the convergence rate at these points will deteriorate when they are close to $x=\pm1$.

In Figure \ref{fig:Super} we illustrate the errors at the superconvergence points $\{x_j^R\}_{j=0}^{n}$ for the function $f(x)=(1-x)^{5/2}e^{x}$ and we choose $n=20$, $\alpha=\beta=0$ and $m=0,1$. Clearly, we see that the error at each point $x_j^R$ is much smaller than the maximum error when $x_j^R$ is not close to the singularity $\xi=1$.

\begin{figure}[htbp]
\centering
\includegraphics[width=0.49\textwidth,height=0.4\textwidth]{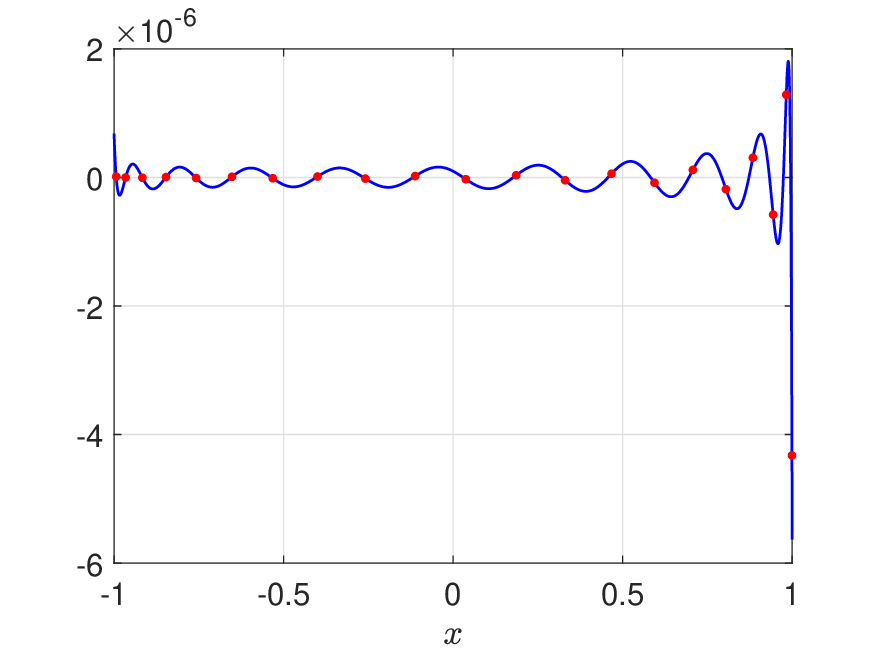}
\includegraphics[width=0.49\textwidth,height=0.4\textwidth]{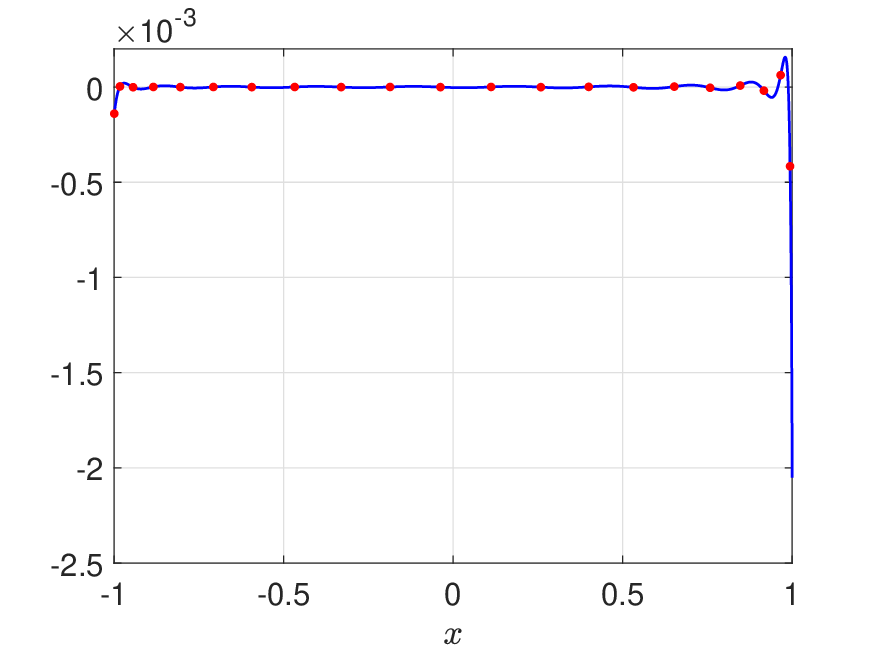}
\caption{Errors of the zero- (left) and first-order (right) Jacobi spectral differentiations with $n=20$ and $\alpha=\beta=0$ for $f(x)=(1-x)^{5/2}e^{x}$. The points indicates the errors at the superconvergence points $\{x_j^R\}_{j=0}^{n}$ .}\label{fig:Super}
\end{figure}

\subsection{Truncated power functions}
We extend our discussion to Jacobi spectral differentiation for some other singular functions.
Consider truncated power functions of the form
\begin{equation}\label{eq:TrunPower}
f(x) = (x-\xi)_{+}^{\sigma} g(x),
\end{equation}
where $\xi\in(-1,1)$ and $(\cdot)_{+}^{\sigma}$ is defined by $(x)_{+}^{\sigma}=x^{\sigma}$ for $x\geq0$ and $(x)_{+}^{\sigma}=0$ for $x<0$. Such functions arise naturally in the remainder of quadrature formulas by Peano kernel theorem and error estimates of their Legendre and Jacobi approximations were studied in \cite{Babuska2019,Gui1986,Xiang2023}. Below we consider the pointwise error estimates of Jacobi spectral differentiations to \eqref{eq:TrunPower}.

\begin{theorem}\label{thm:TrunPower}
Let $f$ be defined in \eqref{eq:TrunPower} and let $\{a_k^{(\alpha,\beta)}\}$ be the Jacobi coefficients of $f$. Moreover, let $\mathcal{R}_{n}^{m}(x)$ denote the remainder of Jacobi spectral differentiation of order $m$ to $f$. Then, the following results hold.
\begin{itemize}
\item[\rm(i)] For $\sigma>-1$, we have
\begin{align}
a_k^{(\alpha,\beta)} &= \mathcal{T}_{\sigma,\xi}^{\alpha,\beta} \frac{\cos(k\arccos(\xi)-\phi_{\alpha,\beta}(\xi))}{k^{\sigma+1/2}} + O\left(\frac{1}{k^{\sigma+3/2}}\right),
\end{align}
where $\phi_{\alpha,\beta}(x) = ((\sigma+\alpha)/2 + 3/4)\pi - (\alpha+\beta+1)\arccos(x)/2$ and
\[
\mathcal{T}_{\sigma,\xi}^{\alpha,\beta} = \frac{\Gamma(\sigma+1) (1-\xi)^{\frac{\sigma+\alpha}{2}+\frac{1}{4}} (1+\xi)^{\frac{\sigma+\beta}{2}+\frac{1}{4}} }{2^{\frac{\alpha+\beta-1}{2}} \sqrt{\pi} } g(\xi).
\]

\item[\rm(ii)] Let $\sigma>0$ be not an even integer and let $m<\min\{\sigma,(\sigma+1/2-\beta)/2,(\sigma+1/2-\alpha)/2\}$, we have for $x\neq\xi$,
\begin{align}
\mathcal{R}_{n}^{m}(x) &= \left\{
            \begin{array}{ll}
{\displaystyle O(n^{2m+\beta-\sigma-1/2})},    & \hbox{$x=-1$,}    \\[5pt]
{\displaystyle O(n^{2m+\alpha-\sigma-1/2})},   & \hbox{$x=1$,}     \\[5pt]
{\displaystyle O(n^{m-\sigma-1}) },   & \hbox{$x\in(-1,\xi)\cup(\xi,1)$,}
            \end{array}
            \right.
\end{align}
and for $x=\xi$,
\begin{align}\label{eq:PointError2}
\mathcal{R}_{n}^{m}(\xi) &= \left\{
            \begin{array}{ll}
{\displaystyle O(n^{m-\sigma-1}) },   & \hbox{$|m-\sigma|$ even,}    \\[5pt]
{\displaystyle O(n^{m-\sigma}) },     & \hbox{otherwise.}
            \end{array}
            \right.
\end{align}
\end{itemize}
\end{theorem}
\begin{proof}
The assertion (i) follows from the proof of Theorem \ref{thm:AsyJac} and the assertion (ii) follows by using a similar proof as that of Theorem \ref{thm:JacobiDiff}. We omit the details.
\end{proof}

\begin{remark}
When the order of differentiation increases from $m$ to $m+1$, we see that the convergence rate of Jacobi differentiation deteriorates two orders at both endpoints and one order at each point in the smooth region. The pattern of deterioration at the singularity $x=\xi$ is more complicated. Specifically, if $\sigma\in\mathbb{N}$, then the convergence rate does not deteriorate if $|m-\sigma|$ is odd and deteriorates two orders if $|m-\sigma|$ is even. If $\sigma\notin\mathbb{N}$, however, the convergence rate at $x=\xi$ always deteriorates one order. Note that the pattern of deterioration of the function \eqref{eq:TrunPower} at the singularity is different from that of the function \eqref{def:ModelFun} when $\sigma\notin\mathbb{N}$.
\end{remark}

\begin{remark}
Babu\v{s}ka and Hakula in \cite{Babuska2019} studied the pointwise error estimate of Legendre approximation, i.e., $\alpha=\beta=0$, and they stated in \cite[Remark~1]{Babuska2019} that ``the theoretical constant $C^{*}(x)$ cannot be computed'', where $C^{*}(x)$ is defined by
\[
C^{*}(x) = \lim_{n\rightarrow\infty} n^{\kappa} \sup\left|f(x) - S_n^{(0,0)}(f;x)\right|,
\]
whenever $|f(x) - S_n^{(0,0)}(x)|=O(n^{-\kappa})$ at the point $x$ as $n\rightarrow\infty$. We point out that the leading term of the asymptotic of the remainder $\mathcal{R}_{n}^{m}(x)$ can be derived following the line of the proof of Theorem \ref{thm:JacobiDiff}, and thus $C^{*}(x)$ can actually be computed. For example, if $|m-\sigma|$ is not even, then for $x=\xi$, after some calculations we have
\begin{align}
\mathcal{R}_{n}^{m}(\xi) &= \frac{2^{\frac{\alpha+\beta-1}{2}} \mathcal{T}_{\sigma,\xi}^{\alpha,\beta} }{\sqrt{\pi \omega_{m+\alpha+1/2,m+\beta+1/2}(x) }} \sin\left(\frac{m-\sigma}{2}\pi\right) \frac{1}{\sigma-m} n^{m-\sigma} + O(n^{m-\sigma-1}). \nonumber
\end{align}
Hence, when $m=0$, then $|f(\xi) - S_n^{(0,0)}(\xi)|=O(n^{-\sigma})$ and
\begin{align}
C^{*}(\xi) &= \lim_{n\rightarrow\infty} n^{\sigma} |\mathcal{R}_{n}^{0}(\xi)| = (1-\xi^2)^{\sigma/2} \frac{\Gamma(\sigma)}{\pi} \left|\sin\left(\frac{\sigma\pi}{2}\right) \right|. \nonumber
\end{align}
\end{remark}

\section{Concluding remarks}\label{sec:conclusion}
We have studied the deterioration of convergence rate of spectral differentiation by Jacobi projection for functions with singularities. We showed that the deteriorations of convergence rate at the endpoints, singularities and other points in the smooth region exhibit different patterns. Extensions to some related problems, including spectral differentiation by Chebyshev interpolation, Lebesgue lemma for Jacobi approximation and superconvergence points of Jacobi approximation and differentiation and Jacobi spectral differentiation for truncated power functions, are discussed. 

Before closing this work, we list several issues for future research:
\begin{itemize}
\item For spectral differentiations using interpolation, we only analyzed the Chebyshev case by using the aliasing formula of Chebyshev coefficients. However, the analysis of other cases, such as Legendre interpolation, still remains open. Note that the aliasing formula is not available for the Legendre coefficients. Moreover, we only considered the case of exact samples. The case that the samples, i.e., $\{f(x_k)\}_{k=0}^{n}$, are polluted by noise is also worthy of further pursuit (see, e.g., \cite{Bruno2012}).

\item It is possible to extend the current analysis to some other singular functions, such as
    \[
    f(x) = |x-\xi|^{\sigma}\log^{\eta}|x-\xi| g(x),
    \]
    where $\sigma>0$ and $\eta\in\mathbb{N}$. Asymptotic estimate of the Jacobi coefficients of this function was analyzed in \cite{Xiang2021}. Similar deterioration results for Jacobi spectral differentiations can be expected. On the other hand, it is also possible to extend the current analysis to Jacobi spectral approximations with $\alpha=\beta\in\{-1,-2,\ldots\}$ and other spectral approximations, such as Laguerre and Hermite spectral approximations and modified Fourier expansion \cite{adcock2010b,Iserles2008}.

\item We have restricted our analysis to the one dimensional case in this work. It is of great interest to extend the analysis to the multivariate case. Note that this issue have received much attention in developing the approximation theory of the $p$-version of finite element method for singular solutions in two and three dimensions (see, e.g., \cite{Babuska2001,Guo2006}).
\end{itemize}
We will address these issues in the future.

\end{document}